\documentclass[11pt]{amsart}
\usepackage[margin=2.5cm]{geometry}
\RequirePackage{amsmath,amssymb,amsthm, amscd, comment,mathtools}
\usepackage{tikz-cd}
\usepackage{extarrows}
\usepackage[pagebackref,breaklinks,colorlinks,linkcolor=red,anchorcolor=red,citecolor=blue]{hyperref}
\usepackage{enumitem}

\newtheorem{theorem}[subsection]{Theorem}
\newtheorem{proposition}[subsection]{Proposition}
\newtheorem{corollary}[subsection]{Corollary}
\newtheorem{lemma}[subsection]{Lemma}
\newtheorem{conjecture}[subsection]{Conjecture}

\theoremstyle{definition}

\newtheorem{definition}[subsection]{Definition}
\newtheorem{assumption}[subsection]{Assumption}
\newtheorem{remark}[subsection]{Remark}

\newtheorem{example}[subsection]{Example}
\numberwithin{equation}{subsection}

\makeatletter

\begin{document}
\title{Additivity of non-acyclicity classes for constructible \'{e}tale sheaves}
\author[Jiangnan Xiong]{Jiangnan Xiong${}^{\dag}$}
\thanks{${}^{\dag}$2401110029@stu.pku.edu.cn}
\thanks{${}^\dag$School of Mathematical Sciences, Peking University, No.5 Yiheyuan Road Haidian District.,
	Beijing, 100871, P.R. China.}
\thanks{\today}
\begin{abstract}
	Using a bivariant version of cohomological correspondences, we establish a categorical trace-like formula for the non-acyclicity classes introduced by Yang and Zhao (Invent. Math. 240: 123–191, 2025). As an application, we prove the additivity for the non-acyclicity classes.
\end{abstract}
\maketitle
\tableofcontents
\section{Introduction}
\subsection{}
Cohomological characteristic class of a constructible \'{e}tale sheaf is an important invariant in geometric ramification theory. Let $h:X\to\mathrm{Spec}k$ be a separated scheme of finite type over a perfect field $k$, $\Lambda$ a finite local ring whose residue field has characteristic invertible in $k$, and $\mathcal{F}$ a constructible complex of $\Lambda$-modules of finite Tor-amplitude on $X$. Abbes and Saito introduced the cohomological characteristic class $C_{X/k}(\mathcal{F})\in H^0(X,\mathcal{K}_{X/k})$ in \cite{AS07} using the Verdier pairing, where $\mathcal{K}_{X/k}=Rh^!\Lambda$ is the dualizing complex.

\subsection{}
Using the singular support defined by Beilinson \cite{Bei16}, Saito constructed the characteristic cycles and the geometric characteristic classes of constructible \'{e}tale sheaves \cite{Sai17}. Saito further proposed the following conjecture:

\begin{conjecture}[{\cite[Conjecture 6.8.1]{Sai17}}]
	Let $X$ be a closed subscheme of a smooth scheme over a perfect field $k$, and $\mathcal{F}$ a constructible complex of $\Lambda$ modules of finite Tor-amplitude on $X$. Let $cc_{X,0}(\mathcal{F})\in\mathrm{CH}_0(X)$ be the geometric characteristic class defined in \cite[Definition 6.7.2]{Sai17} via the characteristic cycle of $\mathcal{F}$. Then we have 
	\begin{equation}
		C_{X/k}(\mathcal{F})=\mathrm{cl}(cc_X(\mathcal{F}))\in H^0(X,\mathcal{K}_{X/k}),
	\end{equation}
	where $\mathrm{cl}:\mathrm{CH}_0(X)\to H^0(X,\mathcal{K}_{X/k})$ is the cycle class map.
\end{conjecture}

Yang and Zhao confirmed the quasi-projective case of Saito's conjecture \cite[Theorem 5.1]{YZ25}.

\subsection{}
While studying Saito's conjecture on characteristic classes, Yang and Zhao realized that it's important to construct a relative version of cohomological characteristic classes. Let $S$ be a noetherian scheme, $h:X\to S$ a separated morphism of finite type, $\Lambda$ a noetherian ring annihilated by some integer invertible on $S$, and $\mathcal{K}_{X/S}=Rh^!\Lambda$ the dualizing complex. Under certain smoothness and transverality conditions, Yang and Zhao introduced the relative cohomological characteristic class $C_{X/S}(\mathcal{F})\in H^0(X,\mathcal{K}_{X/S})$ in \cite{YZ22}.

\subsection{}
Lu and Zheng generalized the construction further to universally locally acyclic (ULA) complexes \cite{LZ22}. They constructed a monoidal category of cohomological correspondences over $S$, whose objects are pairs $(X;\mathcal{F})$, where $X$ is a scheme separated of finite type over $S$, and $\mathcal{F}$ a complex of $\Lambda$-modules on $X$. If $\mathcal{F}$ is constructible of finite Tor-amplitude, Lu and Zheng showed that $(X;\mathcal{F})$ is a dualizable object if and only if $\mathcal{F}$ is $h$-ULA. In this case, the relative cohomological characteristic class is given by the categorical trace:
\begin{equation*}
	(S;\Lambda)\xrightarrow{\mathrm{coev}}\mathcal{H}om_S((X;\mathcal{F}),(X;\mathcal{F}))=(X;\mathcal{F})\otimes_S\mathcal{H}om_S((X;\mathcal{F}),(S;\Lambda))\xrightarrow{\mathrm{ev}}(S;\Lambda).
\end{equation*}

\subsection{}
Yang and Zhao confirmed the quasi-projective case of Saito's conjecture using the fibration method. Consider a diagram of schemes separated of finite type over $S$
\begin{equation*}
	\begin{tikzcd}
		Z\arrow[r,hook,"i"]&X\arrow[r,"f"]\arrow[d,"h"]&Y\arrow[dl,"g"]\\
		&S,
	\end{tikzcd}
\end{equation*}
where $i:Z\hookrightarrow X$ is a closed immersion with open complement $j:U\hookrightarrow X$. Under certain smoothness conditions on $g$ and smallness conditions on $Z$, Yang and Zhao constructed a class $C_{X/Y/S}^Z(\mathcal{F})$ for a constructible complex $\mathcal{F}$ of finite Tor-amplitude on $X$ which is $f$-ULA outside of $Z$ and $h$-ULA \cite[4.3]{YZ25}, called the non-acyclicity class (NA class). They proved that $C_{X/Y/S}^Z(\mathcal{F})$ measures the difference between $C_{X/Y}(\mathcal{F})$ and $C_{X/S}(\mathcal{F})$ in a certain sense \cite[Theorem 5.1]{YZ25}. They also proved that $C_{X/Y/S}^Z(\mathcal{F})$ is a higher dimensional generalization of the classical Swan conductor and Artin conductor \cite[Theorem 6.1 and Theorem 6.5]{YZ25}.

\subsection{}
In this paper, we establish a categorical trace-like formula for NA class.
\begin{proposition}[c.f. \ref{def: non-localized NA class}]\label{prop: main prop on NA class}
	The non-localized NA class $C_{X/Y/S}(\mathcal{F})\in H^0(X,\mathcal{K}_{X/Y/S})$ is given by the following composition in the category $\mathrm{CohCorr}_S$ (c.f. \ref{subsec: CohCorr})
	\begin{equation*}
		(X;\Lambda)\xrightarrow{\mathrm{coev}_{X/Y}}\mathcal{H}om_Y((X;\mathcal{F}),(X;\mathcal{F}))\to\Delta_{Y/S}((X;\mathcal{F})\otimes_SD(\mathcal{F}/S))\xrightarrow{\mathrm{ev}_{X/S}}D(X/Y/S),
	\end{equation*}
	where $D(\mathcal{F}/S)=(X,D_{X/S}(\mathcal{F}))$ is the Verdier dual of $\mathcal{F}$ over $S$, $\Delta_{Y/S}:\mathrm{CohCorr}_{Y\times_SY}\to\mathrm{CohCorr}_Y$ is an enhanced version of the functor $\delta^{\triangle}$ defined in \cite[2.3]{YZ25}, and $D(X/Y/S)=(X,\mathcal{K}_{X/Y/S})$ is given by \cite[(4.4)]{YZ25}.
	
	The NA class $C_{X/Y/S}^Z(\mathcal{F})\in H^0_Z(X,\mathcal{K}_{X/Y/S})$ can be described in a similar fashion (c.f. \ref{def: NA class}).
\end{proposition}

\subsection{}
Our approach to proving \ref{prop: main prop on NA class} is as follows: Reviewing the construction of the NA class, we notice that the base change functor $c$ (c.f. \cite[(2.3)]{YZ25}) can be explained by operations of cohomological correspondences (c.f. \ref{subsec: c as coh corr}). We can then rewrite each step in the construction of the NA class in \cite{YZ25} using the language of cohomological correspondences, and finally obtain \ref{prop: main prop on NA class}.

\subsection{}
In order to handle both the ULA condition over $Y$ and the ULA condition over $S$ categorically, we need to consider the compatibility between the category of cohomological correspondences over $Y$ and the category of cohomological correspondences over $S$. In fact, we can study the compatibility among the categories of cohomological correspondences over a base scheme separated of finite type over $S$ at the same time.

Our approach is to study a bivariant version of cohomological correspondences, which is a symmetric monoidal cocartesian fibration (c.f. \ref{def: BivCohCorr} and \ref{subsec: BivCohCorr to Corr})
\begin{equation*}
	\mathrm{BivCohCorr}_S\to\mathrm{Corr}_S,
\end{equation*}
whose fiber at $Y\in\mathrm{Corr}_S$ is the $\infty$-category $\mathrm{CohCorr}_Y$ of cohomological correspondences over $Y$ (which is an $\infty$-categorical variant of Lu and Zheng's category). $\mathrm{BivCohCorr}_S$ encodes the coherent data of compatibility among $\mathrm{CohCorr}_Y$ when the base scheme $Y$ separated of finite type over $S$ varies.

\subsection{}
We use our new description \ref{prop: main prop on NA class} to establish the additivity of NA class, which justifies that NA class is indeed a characteristic class:
\begin{theorem}[c.f. \ref{prop: additivity of non-localized NA class} and \ref{cor: additivity of NA class}]
	For a distinguished triangle $\mathcal{F}'\to\mathcal{F}\to\mathcal{F}''$ of constructible complexes of finite Tor-amplitude on $X$:
	\begin{enumerate}
		\item If $\mathcal{F}',\mathcal{F}$ and $\mathcal{F}''$ are $h$-ULA, then we have the following additivity in $H^0(X,\mathcal{K}_{X/Y/S})$
		\begin{equation*}
			C_{X/Y/S}(\mathcal{F})=C_{X/Y/S}(\mathcal{F}')+C_{X/Y/S}(\mathcal{F}'').
		\end{equation*}
		\item If in addition $\mathcal{F}',\mathcal{F}$ and $\mathcal{F}''$ are $f$-ULA outside a common $Z$, then we have the additivity of NA classes in $H^0_Z(X,\mathcal{K}_{X/Y/S})$
		\begin{equation*}
			C^Z_{X/Y/S}(\mathcal{F})=C^Z_{X/Y/S}(\mathcal{F}')+C^Z_{X/Y/S}(\mathcal{F}'').
		\end{equation*}
	\end{enumerate}
\end{theorem}

\subsection{}
The corresponding statement for cohomological characteristic class is proved by Jin and Yang in \cite{JY21}. Their proof is a combination of Lu and Zheng's formalism and the following homotopy additivity theorem on categorical traces:
\begin{proposition}\cite[Theorem 1.9]{May01}
	Assume $(\mathcal{C},\otimes,1)$ is a triangulated category equipped with a compatible symmetric monoidal closed model structure (c.f. \cite[Definition 4.1]{May01}). 
	
	Let $(X\to Y\to Z)$ be a distinguished triangle in $\mathcal{C}$. Assume that $X,Y$ and $Z$ are dualizable objects in $\mathcal{C}$. Let $f:X\to X,g:Y\to Y$ be morphisms in $\mathcal{C}$ such that the solid arrows in the following diagram commutes
	\begin{equation*}
		\begin{tikzcd}
			X\arrow[r]\arrow[d,"f"]&Y\arrow[r]\arrow[d,"g"]&Z\arrow[r]\arrow[d,"h",dashed]&X[1]\arrow[d,"{-f[1]}"]\\
			X\arrow[r]&Y\arrow[r]&Z\arrow[r]&X[1].
		\end{tikzcd}
	\end{equation*}
	
	Then there is a map $h:Z\to Z$ filling this diagram, and satisfies the following equality
	\begin{equation*}
		\mathrm{Tr}(g)=\mathrm{Tr}(f)+\mathrm{Tr}(h),
	\end{equation*}
	where $DX$ is the dual of $X$ in $\mathcal{C}$, and $\mathrm{Tr}(f)$ is the categorical trace
	\begin{equation*}
		\mathrm{Tr}(f):1\xrightarrow{\mathrm{coev}}X\otimes DX\xrightarrow{f\otimes\mathrm{id}}X\otimes DX\xrightarrow{\mathrm{ev}}1.
	\end{equation*}
\end{proposition}

In \cite{GPS13}, Groth, Ponto and Schulman proved a similar theorem in the context of monoidal derivators. 

\subsection{}
It is necessary to use homotopy-theoretic or $\infty$-categorical tools (such as model categories or derivators mentioned above). In fact, Ferrand constructed an example \cite{Fer05} of an endomorphism of a distinguished triangle in an (ordinary) derived category, whose traces do not satisfy the additivity relation.

\subsection{}
In the rest of this introduction, we prove the additivity of categorical traces in stable symmetric monoidal $\infty$-categories.

\begin{proposition}\label{prop: Additivity of trace in stable category}
	Let $(\mathcal{C},\otimes,1)$ be a stable symmetric monoidal $\infty$-category. Let $E=(X\to Y\to Z)$ be an exact triangle in $\mathcal{C}$ (c.f. \ref{subsec: exact triangle}), and $\alpha=(f,g,h):E\to E$ be an endomorphism of exact triangle which is depicted as follows
	\begin{equation}
		\begin{tikzcd}
			X\arrow[r]\arrow[d,"f"]&Y\arrow[r]\arrow[d,"g"]&Z\arrow[d,"h"]\\
			X\arrow[r]&Y\arrow[r]&Z.
		\end{tikzcd}
	\end{equation} 
	
	Then we have an equiality in $\pi_0\mathrm{Hom}_{\mathcal{C}}(1,1)$
	\begin{equation*}
		\mathrm{Tr}(g)=\mathrm{Tr}(f)+\mathrm{Tr}(h).
	\end{equation*}
\end{proposition}

\subsection{}
Ramzi proved a similar results using additivity of topological Hochschild homology \cite{Ram22}. Our proof is more explicit, and has the advantage of being applicable to the additivity of NA classes.
Our method is based on a simple observation:
\begin{lemma}[c.f. \ref{prop: nine-diagram to sequence}]
	Let $(\mathcal{C},\otimes,1)$ be a stable symmetric monoidal $\infty$-category. We can associated functorially to every exact nine-diagram (c.f. \ref{def: nine-diagram}), i.e. a diagram
	\begin{equation*}
		\begin{tikzcd}
			X_{0,0}\arrow[r]\arrow[d]&X_{0,1}\arrow[r]\arrow[d]&X_{0,2}\arrow[d]\\
			X_{1,0}\arrow[r]\arrow[d]&X_{1,1}\arrow[r]\arrow[d]&X_{1,2}\arrow[d]\\
			X_{2,0}\arrow[r]&X_{2,1}\arrow[r]&X_{2,2},
		\end{tikzcd}
	\end{equation*}
	whose every columns and every rows are exact triangles, an exact triangle
	\begin{equation*}
		\mathrm{cofib}(d_0)\to C_2\to\mathrm{fib}(d_3),
	\end{equation*}
	where the chain $C_0\xrightarrow{d_0}C_1\xrightarrow{d_1}C_2\xrightarrow{d_2}C_3\xrightarrow{d_3}C_4$ is given by
	\begin{equation*}
		\begin{aligned}
			&C_i=\oplus_{j+k=i}X_{j,k},\\
			&d_i|_{X_{j,k}}=d^h_{j,k}+(-1)^{j+k}d^v_{j,k}.
		\end{aligned}
	\end{equation*}
\end{lemma}

\subsection{}
Given an endomorphism of exact triangle $\alpha:E\to E$ in $\mathcal{C}$, we have the following maps of exact nine-diagrams
\begin{equation*}
	\mathrm{S}(1)\xrightarrow{\mathrm{coev}}\mathcal{H}om(E,E)=E\otimes\mathcal{H}om(E,1)\xrightarrow{\alpha\otimes\mathrm{id}}E\otimes\mathcal{H}om(E,1)\xrightarrow{\mathrm{ev}}\mathrm{T}(1),
\end{equation*}
where $\mathrm{S}(1),\mathrm{T}(1)$ are degenerate exact nine-diagrams (c.f. \ref{ex: source nine-diagram} and \ref{ex: target nine-diagram}), the $\mathrm{coev}$ part and the $\mathrm{ev}$ part are obtained by adjunction, and the isomorphism is obtained by dualizability. The proof is completed by passing to the associated exact triangles.

\subsection{}
Our proof of the additivity of NA classes is essentially the same as the above discussion. In fact, our method can be applied to prove the additivity of any trace-like natural transformation. Since it is not easy to formulate a precise meaning of being trace-like, we won't pursue this point further in this paper.

\subsection{}
The article is organized as follows. In section \ref{sec: NA class}, we introduce a bivariant version of cohomological correspondences, and use it to establish a categorical trace-like formula for the non-acyclicity class (NA class). In section \ref{sec: Nine-diagram}, we discuss a special kind of diagram called exact nine-diagrams, which plays an important role in our proof of the additivity. In section \ref{sec: Additivity}, we prove the additivity of the NA classes.

\subsection*{Acknowledgments} 
The author would like to express his sincere gratitude to Enlin Yang for his careful reading and many improvement suggestions. The author thanks Zhenpeng Li and Xiangyu Pan for their comments.
The author learned the technical proposition \ref{prop: nine-diagram to sequence} from a note of Enlin Yang\footnote{The note is available at Yang's homepage \href{https://www.math.pku.edu.cn/teachers/yangenlin/PFF}{https://www.math.pku.edu.cn/teachers/yangenlin/PFF}.}. This work is partially supported by National Key R\&D program 2021YFA1001400, NSFC general program 12271006 and Beijing Natural Science Foundation QY23005.

\subsection*{Notation and Conventions}
We freely use the language of higher category theory and higher algebra as developed in \cite{Kerodon,HTT,HA}. All of our notions will be assumed to be homotopical or $\infty$-categorical, and we suppress this from our notation. For instance, we use the word \emph{category} to mean $(\infty,1)$-category, the word \emph{isomorphism} to mean a morphism invertible up to homotopy, and the word \emph{unique} to mean unique up to contractible choice.
\begin{enumerate}
	\item We always use the word \emph{diagram} to mean (coherent) commutative diagrams.
	\item Let $\mathrm{Cat}$ be the category of categories.
	\item Let $\Delta=\{[n]=\{0<\cdots<n\}:n\in\mathbb{N}\}\subset\mathrm{Cat}$ be the simplex category.
	\item Let $\square=[1]\times[1]$ be the index category of square diagram.
	\item Let $\mathrm{Hom}_{\mathcal{C}}(X,Y)$ be the space of morphisms from $X$ to $Y$ in $\mathcal{C}$.
	\item Let $\mathrm{Fun}(\mathcal{C},\mathcal{D})$ be the category of functors from $\mathcal{C}$ to $\mathcal{D}$.
	\item Let $\mathrm{Fin}_*$ be the category of pointed finite sets, whose objects are (up to isomorphism) of the form $\langle n\rangle=\{*,1,\cdots,n\}$ with base-point $*$.
	\item For a symmetric monoidal category $\mathcal{C}$, let $\mathcal{C}^{\otimes}\in\mathrm{Cat}_{/\mathrm{Fin}_*}$ be its category of operations.
	\item For functor $F:\mathcal{C}\to\mathrm{Cat}$, let $\int_{\mathcal{C}}F$ be its category of covariant elements, i.e. $\int_{\mathcal{C}}F\to\mathcal{C}$ is the unique cocartesian fibration whose covariant transport functor is $F$.
	\item We use $f^*,f_*,f_!,f^!,\otimes_X,\mathcal{H}om_X$ to denote the six functors. In particular, we omit $L$ or $R$ from the notation, while always dealing with derived functors.
	\item We denote a correspondence $(X\xleftarrow{g}C\xrightarrow{f}Y)$ in $\mathrm{Sch}_S$ by $[g_!f^*]$.
	\item In a diagram of cohomological correspondences, we indicate cocartesian edge by $\twoheadrightarrow$, and cartesian edge by $\hookrightarrow$.
\end{enumerate}

\section{Non-acyclicity classes revisited}\label{sec: NA class}
In this section, we will construct a bivariant version of cohomological correspondences introduced in \cite{LZ22}, and then use it to establish a categorical trace-like formula for the non-acyclicity class (NA class) introduced in \cite{YZ25}.

\subsection{}\label{subsec: CohCorr}
We first recall the category of cohomological correspondences.
\begin{enumerate}[label=(\alph*)]
	\item Let $S$ be a noetherian scheme, $n\in\mathbb{N}$ invertible on $S$, $\Lambda$ a ring with $n\Lambda=0$.
	\item Let $\mathrm{Sch}_S$ be the category of schemes separated of finite type over $S$, $\mathrm{Corr}_S=\mathrm{Corr}(\mathrm{Sch}_S)$ the category of correspondences over $S$.
	\item For convenience, let $\mathcal{D}$ be the six-functor formalism of \'{e}tale cohomology on $S$ with $\Lambda$-coefficient. Most of the arguments in this paper also hold for a general six-functor formalism.
	\item Let $\mathrm{CohCorr}_S=\int_{\mathrm{Corr}_S}\mathcal{D}$ be the symmetric monoidal category of cohomological correspondences over $S$.
\end{enumerate}

\subsection{}
$\mathrm{CohCorr}_S$ can be described as follows:
\begin{enumerate}[label=(\alph*)]
	\item An object is of the form $(X;\mathcal{F})$ where $X\in\mathrm{Sch}_S$, $\mathcal{F}\in\mathcal{D}(X)$. If no confuse will be caused, we will abbreviate $(X;\mathcal{F})$ to $\mathcal{F}$ and $(X;\Lambda)$ to $X$.
	\item A morphism $(X;\mathcal{F})\to(Y;\mathcal{G})$ is given by a correspondence $[g_!f^*]:(X\xleftarrow{f}C\xrightarrow{g}Y)$ together with a map $g_!f^*\mathcal{F}\to\mathcal{G}$ in $\mathcal{D}(Y)$.
	\item The tensor product is given by
	\begin{equation}
		(X;\mathcal{F})\otimes_S(Y;\mathcal{G})=(X\times_SY;\mathcal{F}\boxtimes_S\mathcal{G}=\mathrm{pr}_X^*\mathcal{F}\otimes_{X\times_SY}\mathrm{pr}_Y^*\mathcal{G}).
	\end{equation}
	\item $\mathrm{CohCorr}_S$ has all internal Hom, which is given by
	\begin{equation}
		\mathcal{H}om_S((X;\mathcal{F}),(Y;\mathcal{G}))=(X\times_SY;\mathcal{H}om_{X\times_SY}(\mathrm{pr}_X^*\mathcal{F},\mathrm{pr}_Y^!\mathcal{G})).
	\end{equation}
\end{enumerate}

\begin{definition}
	For scheme $S\in\mathrm{Sch}$, define a symmetric monoidal category $\mathrm{BivCorr}_S$ as follows:
	\begin{enumerate}[label=(\alph*)]
		\item An object is of the form $(X/Y)$ where $X\xrightarrow{f}Y$ a map in $\mathrm{Sch}_S$.
		\item A morphism $[p_!q^*/u_!v^*]:(X/Y)\to(X'/Y')$ is given by a correspondence $[u_!v^*]:Y\to Y'$ together with a compatible correspondence $[p_!q^*]:X\to X'$, i.e. a diagram in $\mathrm{Sch}_S$
		\begin{equation}\label{eq: morphism in BivCorr}
			\begin{tikzcd}
				X\arrow[d]&E\arrow[l,"q"]\arrow[r,"p"]\arrow[d]&X'\arrow[d]\\
				Y&C\arrow[l,"v"]\arrow[r,"u"]&Y'.
			\end{tikzcd}
		\end{equation}
		\item The tensor product is given by
		\begin{equation}
			(X/Y)\otimes(X'/Y')=(X\times_SX'/Y\times_SY').
		\end{equation}
	\end{enumerate}
\end{definition}

\subsection{}
We have the following observations:
\begin{enumerate}[label=(\alph*)]
	\item $\mathrm{BivCorr}_S$ is the category of covariant elements of a lax symmetric monoidal functor
	\begin{equation}
		\mathrm{Corr}_{\bullet/S}:\mathrm{Corr}_S\to\mathrm{Cat},
	\end{equation}
	which sends object $Y\in\mathrm{Corr}_S$ to $\mathrm{Corr}_Y$, and sends a correspondence $[v_!u^*]:Y\to Y'$ to
	\begin{equation}
		v_!u^*:\mathrm{Corr}_Y\to\mathrm{Corr}_{Y'}:(X\to Y)\mapsto(u^*X\to Y').
	\end{equation}
	
	\item $\mathrm{BivCorr}_S$ has internal Hom
	\begin{equation}
		\mathcal{H}om((X/Y),(X'/Y'))=(X\times_SX'/Y\times_SY').
	\end{equation}
	In particular, every object of $\mathrm{BivCorr}_S$ is dual to itself.
\end{enumerate}

\begin{definition}\label{def: BivCohCorr}
	We have a lax symmetric monoidal functor
	\begin{equation}
		\mathrm{BivCorr}_S\to\mathrm{Cat},(X/Y)\mapsto\mathcal{D}(X).
	\end{equation}
	Let $\mathrm{BivCohCorr}_S$ be its category of covariant elements.
\end{definition}

\subsection{}
Unwinding the definitions, we can describe $\mathrm{BivCohCorr}_S$ as follows:
\begin{enumerate}[label=(\alph*)]
	\item An object is $(X/Y;\mathcal{F})$ where $(X\xrightarrow{f}Y)$ a map in $\mathrm{Sch}_S$, $\mathcal{F}\in\mathcal{D}(X)$.
	\item A morphism $(X/Y;\mathcal{F})\to(X'/Y';\mathcal{F}')$ is given by a diagram \eqref{eq: morphism in BivCorr} together with a map $(p_!q^*\mathcal{F}\to\mathcal{F}')$ in $\mathcal{D}(X')$.
	\item The tensor product is given by
	\begin{equation}
		(X/Y;\mathcal{F})\otimes(X'/Y';\mathcal{F}')=(X\times_SX'/Y\times_SY';\mathcal{F}\boxtimes_S\mathcal{F}').
	\end{equation}
	\item The internal Hom is given by
	\begin{equation}
		\mathcal{H}om((X/Y;\mathcal{F}),(X'/Y';\mathcal{F}'))=(X\times_SX'/Y\times_SY';\mathcal{H}om_{X\times_SX'}(\mathrm{pr}_1^*\mathcal{F},\mathrm{pr}_2^!\mathcal{F}')).
	\end{equation}
\end{enumerate}

\subsection{}\label{subsec: BivCohCorr to Corr}
The symmetric monoidal cocartesian fibration $\mathrm{BivCohCorr}_S\to\mathrm{BivCorr}_S\to\mathrm{Corr}_S$ has fiber $\mathrm{CohCorr}_Y$ at $Y$. We can then regard an object $(X;\mathcal{F})\in\mathrm{CohCorr}_S$ together with a structure map $X\to Y$ in $\mathrm{Sch}_S$ as an object $(X/Y;\mathcal{F})\in\mathrm{BivCohCorr}_S$. We often use this point of view implicitly.

\subsection{}\label{subsec: cov trans of BivCohCorr to Corr}
Many operations are covariant transport of $\mathrm{BivCohCorr}_S^{\otimes}\to\mathrm{Corr}_S^{\otimes}$.

\begin{enumerate}[label=(\alph*)]
	\item The exterior product is covariant transport over $(Y,Y')\twoheadrightarrow Y\times_SY'$
	\begin{equation}
		(X;\mathcal{F})\boxtimes(X';\mathcal{F}')=(X\times_SX',\mathcal{F}\boxtimes_S\mathcal{F}').
	\end{equation}
	\item The tensor product over $Y$ is covariant transport over $[(\mathrm{id},\mathrm{id})^*]:(Y,Y)\to Y$
	\begin{equation}
		(X;\mathcal{F})\otimes_Y(X';\mathcal{F}')=(X\times_YX';\mathcal{F}\boxtimes_Y\mathcal{F}').
	\end{equation}
	\item For $f:Y'\to Y$ in $\mathrm{Sch}_S$, the covariant transport over $[f^*]:Y\to Y'$ is given by
	\begin{equation}
		f^*(X;\mathcal{F})=(X;\mathcal{F})\otimes_YY'=(X';(f')^*\mathcal{F}),
	\end{equation}
	where $f':X'=X\times_YY'\to X$ is the base change of $f$.
	\item The covariant transport over $[f_!]$ is given by
	\begin{equation}
		f_!(X;\mathcal{}F)=(X;\mathcal{F}),
	\end{equation}
	where only base scheme is changed. We often apply this functor implicitly to regard an object of $\mathrm{CohCorr}_{Y'}$ as an object of $\mathrm{CohCorr}_Y$.
\end{enumerate}

\subsection{}
By \ref{subsec: cov trans of BivCohCorr to Corr}, for $\mathcal{F},\mathcal{G}\in\mathrm{CohCorr}_Y$, we have canonical natural isomorphisms:
\begin{align}
	&\mathcal{F}\otimes_Y\mathcal{G}=\mathcal{F}\otimes_S\mathcal{G}\otimes_{Y\times_SY}Y,\\
	&\mathcal{H}om_Y(\mathcal{F},\mathcal{G})=\mathcal{H}om_{Y\times_SY}(Y,\mathcal{H}om_S(\mathcal{F},\mathcal{G})),\\
	&\mathcal{H}om_S(\mathcal{F},\mathcal{G})=\mathcal{H}om_Y(\mathcal{F},\mathcal{H}om_S(Y,\mathcal{G})).
\end{align}

We will use these identifications repeatedly in this paper.

\subsection{}
We recall a standard construction in \cite{YZ25}. Consider a pullback diagram in $\mathrm{Sch}_S$:
\begin{equation}
	\begin{tikzcd}
		W\arrow[r,"i"]\arrow[d,"p"]\arrow[dr,very near start,phantom,"\lrcorner"]&X\arrow[d,"f"]\\
		Z\arrow[r,"\delta"]&Y.
	\end{tikzcd}
\end{equation}

For $\mathcal{F}\in\mathcal{D}(X),\mathcal{G}\in\mathcal{D}(Y)$, there is a canonical map \cite[(2.3)]{YZ25}
\begin{equation}\label{eq: map c}
	c_{\delta,f,\mathcal{F},\mathcal{G}}:i^*\mathcal{F}\otimes_Wp^*\delta^!\mathcal{G}\to i^!(\mathcal{F}\otimes_Xf^*\mathcal{G}),
\end{equation}
which is obtained by the following composition
\begin{equation}
	\begin{tikzcd}
		i^*\mathcal{F}\otimes_Wp^*\delta^!\mathcal{G}\xrightarrow{\mathrm{unit}}i^!i_!(i^*\mathcal{F}\otimes_Wp^*\delta^!\mathcal{G})=i^!(\mathcal{F}\otimes_Xf^*\delta_!\delta^!\mathcal{G})\xrightarrow{\mathrm{counit}}i^!(\mathcal{F}\otimes_Xf^*\mathcal{G}).
	\end{tikzcd}
\end{equation}

When $\mathcal{F}=\Lambda$, \eqref{eq: map c} recovers the base change map
\begin{equation}
	p^*\delta^!\mathcal{G}\to i^!f^*\mathcal{G}.
\end{equation}

When $\mathcal{G}=\Lambda$, \eqref{eq: map c} recovers the relative purity map
\begin{equation}
	i^*\mathcal{F}\otimes_Wp^*\delta^!\Lambda\to i^!\mathcal{F}.
\end{equation}

\subsection{}\label{subsec: c as coh corr}
For $(X;\mathcal{F}),(Z;\mathcal{G})\in\mathrm{CohCorr}_Y$, we have a natural map
\begin{equation}\label{eq: c as coh corr}
	\mathcal{F}\otimes_Y\mathcal{H}om_Y(\mathcal{G},Y)\to\mathcal{H}om_Y(\mathcal{G},\mathcal{F}\otimes_YY)=\mathcal{H}om_Y(\mathcal{G},\mathcal{F}),
\end{equation}
which is obtained by adjunction from
\begin{equation}
	\mathcal{F}\otimes_Y\mathcal{G}\otimes_Y\mathcal{H}om_Y(\mathcal{G},Y)\xrightarrow{\mathrm{ev}}\mathcal{F}\otimes_YY=\mathcal{F}.
\end{equation}

Unwinding the definition, \eqref{eq: c as coh corr} recovers the map \eqref{eq: map c}
\begin{equation}
	c:(W;i^*\mathcal{F}\otimes_Wp^*\delta^!\mathcal{G})\to(W;i^!(\mathcal{F}\otimes_Xf^*\mathcal{G})).
\end{equation}

\subsection{}
Recall that $\forall X\in\mathrm{Sch}_S$, $\mathcal{D}(X)$ is a presentable stable symmetric monoidal category. An exact sequence in $\mathrm{CohCorr}_Y$ or in $\mathrm{BivCohCorr}_S$ is an exact sequence in some fiber $\mathcal{D}(X)$. The operation ``taking cofiber'' is understood similarly.

\begin{definition}\label{def: triangle functor}
	For $f:X\to Y$ in $\mathrm{Sch}_S$, $\mathcal{F}\in\mathcal{D}(X)$, define $D(\mathcal{F}/Y)=\mathcal{H}om_Y(\mathcal{F},Y)\in\mathrm{CohCorr}_S$. Use the notation $\mathcal{K}_{X/Y}=f^!\Lambda$ as in \cite{YZ25}, we have $D(X/Y)=(X,\mathcal{K}_{X/Y})$, and $D(\mathcal{F}/Y)$ recovers the Verdier dual
	\begin{equation}\label{eq: D_F/Y}
		D(\mathcal{F}/Y)=(X,\mathcal{H}om_X(\mathcal{F},\mathcal{K}_{X/Y})).
	\end{equation}
	
	We have a natural transformation in $\mathrm{Fun}(\mathrm{CohCorr}_{Y\times_SY},\mathrm{CohCorr}_Y)$
	\begin{equation}
		\bullet\otimes_{Y\times_SY}D(Y/Y\times_SY)\xrightarrow{\eqref{eq: c as coh corr}}\mathcal{H}om_{Y\times_SY}(Y,\bullet).
	\end{equation}
	
	Define a functor $\Delta_{Y/S}:\mathrm{CohCorr}_{Y\times_SY}\to\mathrm{CohCorr}_Y$ as the following cofiber
	\begin{equation}
		\Delta_{Y/S}(\bullet)=\mathrm{cofib}(\bullet\otimes_{Y\times_SY}D(Y/Y\times_SY)\to\mathcal{H}om_{Y\times_SY}(Y,\bullet)).
	\end{equation}
\end{definition}

\subsection{}
By \ref{subsec: c as coh corr}, for every pullback diagram in $\mathrm{Sch}_S$
\begin{equation}
	\begin{tikzcd}
		W\arrow[d]\arrow[r]\arrow[dr,phantom,very near start,"\lrcorner"]&Z\arrow[d]\\
		Y\arrow[r,"\delta_g"]&Y\times_SY,
	\end{tikzcd}
\end{equation}
the functor $\Delta_{Y/S}:\mathcal{D}(Z)\to\mathcal{D}(W)$ recovers the functor $(\delta_g)^{\Delta}$ constructed in \cite[2.3]{YZ25}.
	
For $h:X\xrightarrow{f}Y\xrightarrow{g}S$ in $\mathrm{Sch}_S$, we often use the following pullback diagrams in $\mathrm{Sch}_S$
\begin{equation}
	\begin{tikzcd}
		X\arrow[r,equal]\arrow[d,"\delta_f"]\arrow[dr,phantom,very near start,"\lrcorner"]&X\arrow[d,"\delta_h"]\\
		X\times_YX\arrow[r,"i_f"]\arrow[d,"p_f"]\arrow[dr,phantom,very near start,"\lrcorner"]&X\times_SX\arrow[d,"f\times f"]\\
		Y\arrow[r,"\delta_g"]&Y\times_SY.
	\end{tikzcd}
\end{equation}

\subsection{}
Let $(h:X\to S)\in\mathrm{Sch}_S$, $\mathcal{F}\in\mathcal{D}(X)$. We have coevaluation map in $\mathrm{CohCorr}_S$
\begin{equation}
	S\xrightarrow{\mathrm{coev}_{\mathcal{F}}}\mathcal{H}om_S(\mathcal{F},\mathcal{F}),
\end{equation}
which lies over $[(\delta_h)_!h^*]:S\to X\times_SX$, and hence factors through the cocartesian edge over $[h^*]$:
\begin{equation}
	\begin{tikzcd}
		S\arrow[r,"{[h^*]}",two heads]\arrow[dr,"{\mathrm{coev}_{\mathcal{F}}}"swap]&X\arrow[d,dashed,"{\mathrm{coev}_{\mathcal{F},X/S}}"]\\
		&\mathcal{H}om_S(\mathcal{F},\mathcal{F}).
	\end{tikzcd}
\end{equation}

Similarly, we have evaluation map in $\mathrm{CohCorr}_S$
\begin{equation}
	\mathcal{F}\otimes_S\mathcal{H}om_S(\mathcal{F},S)\to S,
\end{equation}
which lies over $[h_!(\delta_h)^*]:X\times_SX\to S$, and hence factors through the cartesian edge over $[h_!]$:
\begin{equation}
	\begin{tikzcd}
		\mathcal{F}\otimes_SD(\mathcal{F}/S)\arrow[r,dashed,"{\mathrm{ev}_{\mathcal{F},X/S}}"]\arrow[dr,"{\mathrm{ev}_{\mathcal{F}}}"swap]&D(X/S)\arrow[d,hook,"{[h_!]}"]\\
		&S.
	\end{tikzcd}
\end{equation}

\subsection{}
Let $(h:X\to S)\in\mathrm{Sch}_S$, $\mathcal{F}\in\mathcal{D}(X)$.
By \cite{LZ22}, for constructible $\mathcal{F}$ of finite Tor-amplitude, $\mathcal{F}$ is universally locally acyclic over $S$ if and only if $(X;\mathcal{F})$ is dualizable in $\mathrm{CohCorr}_S$, if and only if the following canonical isomorphism in $\mathrm{CohCorr}_S$ is an isomorphism
\begin{equation}\label{eq: dualizable in CohCorr}
	\mathcal{F}\otimes_SD(\mathcal{F}/S)\xrightarrow{=}\mathcal{H}om_S(\mathcal{F},\mathcal{F}).
\end{equation}

For $\mathcal{F}\in\mathrm{CohCorr}_S$ dualizable, the relative cohomological characteristic class is given by the categorical trace
\begin{equation}
	S\xrightarrow{\mathrm{coev}_{\mathcal{F}}}\mathcal{H}om_S(\mathcal{F},\mathcal{F})=\mathcal{F}\otimes_SD(\mathcal{F}/S)\xrightarrow{\mathrm{ev}_{\mathcal{F}}}S.
\end{equation}

More precisely, the composition
\begin{equation}\label{eq: coh char class}
	C_{X/S}(\mathcal{F}):X\xrightarrow{\mathrm{coev}_{\mathcal{F},X/S}}\mathcal{H}om_S(\mathcal{F},\mathcal{F})=\mathcal{F}\otimes_SD(\mathcal{F}/S)\xrightarrow{\mathrm{ev}_{\mathcal{F},X/S}}D(X/S)
\end{equation}
lies in the fiber $\mathcal{D}(X)=\mathrm{CohCorr}_S\times_{\mathrm{Corr}_S}\{X\}$, and is equal to the class $C_{X/S}(\mathcal{F})\in H^0(X,\mathcal{K}_{X/S})$.

\subsection{}
Let $g:Y\to S$ in $\mathrm{Sch}_S$, $\mathcal{F}\in\mathrm{CohCorr}_Y$. We have natural comparison maps
\begin{align}
	\label{eq: from D_F/Y to D_F/S}&D(\mathcal{F}/Y)\otimes_YD(Y/S)\to D(\mathcal{F}/S),\\
	\label{eq: from D_F/S to D_F/Y}&D(\mathcal{F}/S)\otimes_YD(Y/Y\times_SY)\to D(\mathcal{F}/Y),
\end{align}
where \eqref{eq: from D_F/Y to D_F/S} is the natural map
\begin{equation}
	\mathcal{H}om_Y(\mathcal{F},Y)\otimes_YD(Y/S)\to\mathcal{H}om_Y(\mathcal{F},D(Y/S))=D(\mathcal{F}/S),
\end{equation}
and \eqref{eq: from D_F/S to D_F/Y} is given by the following composition
\begin{equation}
	\begin{aligned}
		&D(\mathcal{F}/S)\otimes_YD(Y/Y\times_SY)\\
		=\,&(Y\otimes_SD(\mathcal{F}/S))\otimes_{Y\times_SY}D(Y/Y\times_SY)\\
		\to&\mathcal{H}om_S(\mathcal{F},Y)\otimes_{Y\times_SY}D(Y/Y\times_SY)\\
		\to&\mathcal{H}om_{Y\times_SY}(Y,\mathcal{H}om_S(\mathcal{F},Y))\\
		=\,&\mathcal{H}om_Y(\mathcal{F},Y)=D(\mathcal{F}/Y).
	\end{aligned}
\end{equation}

\subsection{}\label{subsec: from D_F/S to D_F/Y expansion}
Unwinding the definitions, \eqref{eq: from D_F/S to D_F/Y} is obtained by adjunction from the following composition
\begin{equation}
	\begin{aligned}
		&\mathcal{F}\otimes_YD(\mathcal{F}/S)\otimes_YD(Y/Y\times_SY)\\
		=\,\ &\mathcal{F}\otimes_SD(\mathcal{F}/S)\otimes_{Y\times_SY}Y\otimes_YD(Y/Y\times_SY)\\
		\xrightarrow{\mathrm{ev}_Y}&\mathcal{F}\otimes_SD(\mathcal{F}/S)\otimes_Y(Y\times_SY)\\
		=\,\ &\mathcal{F}\otimes_SD(\mathcal{F}/S)\otimes_SY\\
		\xrightarrow{\mathrm{ev}_{\mathcal{F}}}&S\otimes_SY=Y.
	\end{aligned}
\end{equation}

\subsection{}
Consider the special case $\mathcal{F}=Y$. In this case, \eqref{eq: from D_F/Y to D_F/S} is the identity map, and \eqref{eq: from D_F/S to D_F/Y} recovers the bi-evaluation map \cite[(2.26)]{YZ25}
\begin{equation}\label{eq: bi-ev Y/S}
	D(Y/S)\otimes_YD(Y/Y\times_SY)\to Y.
\end{equation}

\begin{lemma}\label{lem: recover D_F/S to D_F/Y from bi-ev}
	We have a natural diagram
	\begin{equation}
		\begin{tikzcd}
			D(\mathcal{F}/S)\otimes_YD(Y/Y\times_SY)\arrow[r,"\eqref{eq: from D_F/S to D_F/Y}"]\arrow[d,equal]&D(\mathcal{F}/Y)\arrow[dd,equal]\\
			\mathcal{H}om_Y(\mathcal{F},D(Y/S))\otimes_YD(Y/Y\times_SY)\arrow[d]\\
			\mathcal{H}om_Y(\mathcal{F},D(Y/S)\otimes_YD(Y/Y\times_SY))\arrow[r,"\eqref{eq: bi-ev Y/S}"]&\mathcal{H}om_Y(\mathcal{F},Y).
		\end{tikzcd}
	\end{equation}
	
	In particular, \eqref{eq: from D_F/S to D_F/Y} recovers \cite[(2.32)]{YZ25}.
\end{lemma}
\begin{proof}
	By adjunction and \ref{subsec: from D_F/S to D_F/Y expansion}, it suffices to note that $\mathrm{ev}_{\mathcal{F}}:\mathcal{F}\otimes_SD(\mathcal{F}/S)\to S$ factors through $Y\otimes_SD(Y/S)$, which is obtained by the following diagram:
	\begin{equation}
		\begin{tikzcd}
			\mathcal{F}\otimes_SD(\mathcal{F}/S)\arrow[d,"\mathrm{ev}_{\mathcal{F},Y/S}"swap]\arrow[dr,dashed]\arrow[rr,"\mathrm{ev}_{\mathcal{F}}"]&&S\\
			D(Y/S)\arrow[r,two heads,"{[\mathrm{pr}_{g,2}^*]}"swap]&Y\otimes_SD(Y/S)\arrow[r,two heads,"{[\delta_g^*]}"swap]\arrow[ur,"\mathrm{ev}_{Y}",near start]&D(Y/S)\arrow[u,hook,"{[g_!]}"swap].
		\end{tikzcd}
	\end{equation}
	
	Here we use the following standard notations
	\begin{equation}
		\begin{tikzcd}
			Y\arrow[r,"\delta_g",hook]\arrow[dr,"\mathrm{id}"swap]&Y\times_SY\arrow[dr,phantom,"\lrcorner",very near start]\arrow[d,"\mathrm{pr}_{g,2}"swap]\arrow[r,"\mathrm{pr}_{g,1}"]&Y\arrow[d,"g"]\\
			&Y\arrow[r,"g"]&S.
		\end{tikzcd}
	\end{equation}
\end{proof}

\begin{proposition}\label{prop: compatible D_F/Y and D_F/S}
	$\forall\mathcal{F}\in\mathrm{CohCorr}_Y$, we have natural diagrams
	\begin{equation}\label{eq: D_F/Y to D_F/S to D_F/Y}
		\begin{tikzcd}
			D(\mathcal{F}/Y)\otimes_YD(Y/S)\otimes_YD(Y/Y\times_SY)\arrow[r,"\eqref{eq: from D_F/Y to D_F/S}"]\arrow[d,"\eqref{eq: bi-ev Y/S}"]&D(\mathcal{F}/S)\otimes_YD(Y/Y\times_SY)\arrow[d,"\eqref{eq: from D_F/S to D_F/Y}"]\\
			D(\mathcal{F}/Y)\otimes_YY\arrow[r,equal]&D(\mathcal{F}/Y),
		\end{tikzcd}
	\end{equation}
	\begin{equation}\label{eq: D_F/S to D_F/Y to D_F/S}
		\begin{tikzcd}
			D(\mathcal{F}/S)\otimes_YD(Y/Y\times_SY)\otimes_YD(Y/S)\arrow[r,"\eqref{eq: from D_F/S to D_F/Y}"]\arrow[d,"\eqref{eq: bi-ev Y/S}"]&D(\mathcal{F}/Y)\otimes_YD(Y/S)\arrow[d,"\eqref{eq: from D_F/Y to D_F/S}"]\\
			D(\mathcal{F}/S)\otimes_YY\arrow[r,equal]&D(\mathcal{F}/S).
		\end{tikzcd}
	\end{equation}
\end{proposition}
\begin{proof}
	By \ref{lem: recover D_F/S to D_F/Y from bi-ev}, \eqref{eq: D_F/Y to D_F/S to D_F/Y} can be obtained from the following diagram
	\begin{equation}
		\begin{small}
			\begin{tikzcd}
				\mathcal{H}om_Y(\mathcal{F},Y)\otimes_YD(Y/S)\otimes_YD(Y/Y\times_SY)\arrow[r,"\eqref{eq: from D_F/Y to D_F/S}"]\arrow[d,equal]&\mathcal{H}om_Y(\mathcal{F},D(Y/S))\otimes_YD(Y/Y\times_SY)\arrow[d]\\
				\mathcal{H}om_Y(\mathcal{F},Y)\otimes_YD(Y/S)\otimes_YD(Y/Y\times_SY)\arrow[r]\arrow[d,"\eqref{eq: bi-ev Y/S}"]&\mathcal{H}om_Y(\mathcal{F},D(Y/S)\otimes_YD(Y/Y\times_SY))\arrow[d,"\eqref{eq: bi-ev Y/S}"]\\
				\mathcal{H}om_Y(\mathcal{F},Y)\otimes_YY\arrow[r,"="]&\mathcal{H}om_Y(\mathcal{F},Y).
			\end{tikzcd}
		\end{small}
	\end{equation}
	
	Also by lemma \ref{lem: recover D_F/S to D_F/Y from bi-ev}, \eqref{eq: D_F/S to D_F/Y to D_F/S} can be obtained from the following diagram
	\begin{equation}
		\begin{tikzcd}
			\mathcal{H}om_Y(\mathcal{F},D(Y/S)\otimes_YD(Y/Y\times_SY))\otimes_YD(Y/S)\arrow[d]\arrow[r,"\eqref{eq: bi-ev Y/S}"]&D(\mathcal{F}/Y)\otimes_YD(Y/S)\arrow[d]\\
			\mathcal{H}om_Y(\mathcal{F},D(Y/S)\otimes_YD(Y/Y\times_SY)\otimes_YD(Y/S))\arrow[r,"\eqref{eq: bi-ev Y/S}"]\arrow[d,"\eqref{eq: bi-ev Y/S}"]&\mathcal{H}om_Y(\mathcal{F},Y\otimes_YD(Y/S))\arrow[d,equal]\\
			\mathcal{H}om_Y(\mathcal{F},D(Y/S)\otimes_YY)\arrow[r,equal]&D(\mathcal{F}/Y).
		\end{tikzcd}
	\end{equation}
\end{proof}

\begin{assumption}\label{assum: smooth assumption}
	In the rest of this paper, we always assume that $D(Y/S)\in\mathrm{CohCorr}_Y$ is a $\otimes_Y$-invertible object, and \eqref{eq: bi-ev Y/S} is an isomorphism. This is satisfied for example when $g:Y\to S$ is cohomologically smooth. Note that by \ref{prop: compatible D_F/Y and D_F/S}, \eqref{eq: from D_F/Y to D_F/S} and \eqref{eq: from D_F/S to D_F/Y} are isomorphisms in this case.
\end{assumption}

\begin{lemma}
	For $\mathcal{F},\mathcal{G}\in\mathrm{CohCorr}_Y$, we have a natural diagram
	\begin{equation}\label{eq: compatible dualizability}
		\begin{tikzcd}
			\mathcal{F}\otimes_YD(\mathcal{G}/Y)\arrow[r]\arrow[d]&\mathcal{H}om_Y(\mathcal{G},\mathcal{F})\arrow[d,equal]\\
			\mathcal{H}om_{Y\times_SY}(Y,\mathcal{F}\otimes_SD(\mathcal{G}/S))\arrow[r]&\mathcal{H}om_{Y\times_SY}(Y,\mathcal{H}om_S(\mathcal{G},\mathcal{F})),
		\end{tikzcd}
	\end{equation}
	where the left map is given by the following composition
	\begin{equation}
		\begin{aligned}
			&\mathcal{F}\otimes_YD(\mathcal{G}/Y)\\
			\xlongequal{\eqref{eq: from D_F/S to D_F/Y}}&\mathcal{F}\otimes_YD(\mathcal{G}/S)\otimes_YD(Y/Y\times_SY)\\
			\xlongequal{\quad\quad\ \,}&\mathcal{F}\otimes_SD(\mathcal{G}/S)\otimes_{Y\times_SY}D(Y/Y\times_SY)\\
			\xrightarrow{\qquad\ \ }&\mathcal{H}om_{Y\times_SY}(Y,\mathcal{F}\otimes_SD(\mathcal{G}/S)).
		\end{aligned}
	\end{equation}
\end{lemma}
\begin{proof}
	It suffices to construct a natural diagram
	\begin{equation}
		\begin{tikzcd}
			\mathcal{F}\otimes_SD(\mathcal{G}/S)\otimes_{Y\times_SY}D(Y/Y\times_SY)\arrow[r,equal]\arrow[d]&\mathcal{F}\otimes_YD(\mathcal{G}/S)\otimes_YD(Y/Y\times_SY)\arrow[d]\\
			\mathcal{H}om_{Y\times_SY}(Y,\mathcal{F}\otimes_SD(\mathcal{G}/S))\arrow[d]&\mathcal{F}\otimes_YD(\mathcal{G}/Y)\arrow[d]\\
			\mathcal{H}om_{Y\times_SY}(Y,\mathcal{H}om_S(\mathcal{G},\mathcal{F}))\arrow[r,equal]&\mathcal{H}om_Y(\mathcal{G},\mathcal{F}).
		\end{tikzcd}
	\end{equation}
	
	By \ref{subsec: from D_F/S to D_F/Y expansion}, both paths are obtained by adjunction from
	\begin{equation}
		\begin{aligned}
			&\mathcal{F}\otimes_S\mathcal{G}\otimes_SD(\mathcal{G}/S)\otimes_{Y\times_SY}Y\otimes_{Y\times_SY}D(Y/Y\times_SY)\\
			\xrightarrow{\mathrm{ev}_{\mathcal{G}}\otimes\mathrm{ev}_{Y}}&\mathcal{F}\otimes_SS\otimes_{Y\times_SY}(Y\times_SY)=Y.
		\end{aligned}
	\end{equation}
\end{proof}

\begin{lemma}\label{lem: compatible coev}
	For $\mathcal{F}$ in $\mathrm{CohCorr}_Y$, we have diagrams
	\begin{equation}\label{eq: compatible coev}
		\begin{tikzcd}
			X\arrow[r,equal]\arrow[d,"\mathrm{coev}_{\mathcal{F},X/Y}"]&\mathcal{H}om_{Y\times_S Y}(Y,X)\arrow[d,"\mathrm{coev}_{\mathcal{F},X/S}"]\\
			\mathcal{H}om_Y(\mathcal{F},\mathcal{F})\arrow[r,equal]&\mathcal{H}om_{Y\times_SY}(Y,\mathcal{H}om_S(\mathcal{F},\mathcal{F})).
		\end{tikzcd}
	\end{equation}
	\begin{equation}\label{eq: compatible ev}
		\begin{tikzcd}
			\mathcal{F}\otimes_YD(\mathcal{F}/S)\otimes_YD(Y/Y\times_SY)\arrow[d,"\mathrm{ev}_{\mathcal{F},X/S}"]\arrow[r,"="]&\mathcal{F}\otimes_YD(\mathcal{F}/Y)\arrow[d,"\mathrm{ev}_{\mathcal{F},X/Y}"]\\
			D(X/S)\otimes_{Y\times_SY}D(Y/Y\times_SY)\arrow[r,"="]&D(X/Y).
		\end{tikzcd}
	\end{equation}
\end{lemma}
\begin{proof}
	\eqref{eq: compatible coev} is obtained by adjunction from
	\begin{equation}
		\begin{tikzcd}
			\mathcal{F}\otimes_YX\arrow[r,equal]\arrow[dr,two heads,"{[(\delta_f)^*]}"swap]&\mathcal{F}\otimes_SX\otimes_{Y\times_SY}Y\arrow[r,equal]&\mathcal{F}\otimes_SX\arrow[dl,two heads,"{[(\delta_h)^*]}"]\\
			&\mathcal{F}.
		\end{tikzcd}
	\end{equation}
	
	By \ref{subsec: from D_F/S to D_F/Y expansion}, \eqref{eq: compatible ev} is obtained by adjunction from the diagram
	\begin{equation}
		\begin{tikzcd}
			X\otimes_Y\mathcal{F}\otimes_YD(\mathcal{F}/S)\otimes_YD(Y/Y\times_SY)\arrow[d,equal]\arrow[r,equal]&X\otimes_Y\mathcal{F}\otimes_YD(\mathcal{F}/Y)\arrow[d,"\mathrm{ev}"]\\
			X\otimes_Y\mathcal{F}\otimes_SD(\mathcal{F}/S)\otimes_YY\otimes_{Y\times_SY}D(Y/Y\times_SY)\arrow[r,"\mathrm{ev}"]&X.
		\end{tikzcd}
	\end{equation}
\end{proof}

\begin{proposition}\cite[Theorem 3.7]{YZ25}\label{prop: fibration formula in transversal case}
	Assume $\mathcal{F}\in\mathcal{D}(X)$ is dualizable both in $\mathrm{CohCorr}_Y$ and in $\mathrm{CohCorr}_S$, then we have the fibration formula in transversal case
	\begin{equation}
		\begin{tikzcd}
			X\arrow[d,"C_{X/Y}(\mathcal{F})"swap]\arrow[dr,"C_{X/S}(\mathcal{F})"]\\
			D(X/Y)\arrow[r,"\eqref{eq: triangle D_F/Y/S}"swap]&D(X/S).
		\end{tikzcd}
	\end{equation}
\end{proposition}
\begin{proof}
	We have the following diagram
	\begin{equation}
		\begin{tikzcd}
			X\arrow[dr,phantom,"\eqref{eq: compatible coev}"]\arrow[r,equal]\arrow[d,"\mathrm{coev}_{\mathcal{F},X/Y}"swap]&\mathcal{H}om_{Y\times_SY}(Y,X)\arrow[d,"\mathrm{coev}_{\mathcal{F},X/S}"]\arrow[r,equal]&X\arrow[d,"\mathrm{coev}_{\mathcal{F},X/S}"]\\
			\mathcal{H}om_Y(\mathcal{F},\mathcal{F})\arrow[r,equal]\arrow[rd,phantom,"\eqref{eq: compatible dualizability}"]&\mathcal{H}om_{Y\times_SY}(Y,\mathcal{H}om_S(\mathcal{F},\mathcal{F}))\arrow[r,hook]&\mathcal{H}om_S(\mathcal{F},\mathcal{F})\\
			\mathcal{F}\otimes_YD(\mathcal{F}/Y)\arrow[dr,phantom,"\eqref{eq: compatible ev}"]\arrow[u,"="]\arrow[d,"\mathrm{ev}_{\mathcal{F},X/Y}"swap]\arrow[r]&\mathcal{H}om_{Y\times_SY}(Y,\mathcal{F}\otimes_SD(\mathcal{F}/S))\arrow[u,"="swap]\arrow[d,"\mathrm{ev}_{\mathcal{F},X/S}"]\arrow[r,hook]&\mathcal{F}\otimes_SD(\mathcal{F}/S)\arrow[u,"="swap]\arrow[d,"\mathrm{ev}_{\mathcal{F},X/S}"]\\
			D(X/Y)\arrow[r]&\mathcal{H}om_{Y\times_SY}(Y,D(X/S))\arrow[r,equal]&D(X/S).
		\end{tikzcd}
	\end{equation}
\end{proof}

\begin{definition}\label{def: triangle D_F/Y/S}
	For $\mathcal{F}\in\mathrm{CohCorr}_Y$, define an exact triangle
	\begin{equation}
		\begin{tikzcd}
			D(\mathcal{F}/Y)\arrow[d,equal]\arrow[r]& D(\mathcal{F}/S)\arrow[d,equal]\arrow[r]&D(\mathcal{F}/Y/S)\arrow[d,equal]\\
			D(\mathcal{F}/S)\otimes_YD(Y/Y\times_SY)\arrow[r]&\mathcal{H}om_{Y\times_SY}(Y,D(\mathcal{F}/S))\arrow[r]&\Delta_{Y/S}D(\mathcal{F}/S).
		\end{tikzcd}
	\end{equation}
\end{definition}

\subsection{}
For $\mathcal{F}=X$, \ref{def: triangle D_F/Y/S} recovers \cite[(4.6)]{YZ25}
\begin{equation}\label{eq: triangle D_F/Y/S}
	\begin{tikzcd}
		D(X/Y)\arrow[d,equal]\arrow[r]&D(X/S)\arrow[d,equal]\arrow[r]&D(X/Y/S)\arrow[d,equal]\\
		(X,\mathcal{K}_{X/Y})\arrow[r,"(\delta_g)^!"]&(X,\mathcal{K}_{X/S})\arrow[r]&(X,\mathcal{K}_{X/Y/S}).
	\end{tikzcd}
\end{equation}

\begin{definition}\label{def: non-localized NA class}
	Assume $(X;\mathcal{F})\in\mathrm{CohCorr}_Y$ is dualizable in $\mathrm{CohCorr}_S$.
	
	By proof of \ref{prop: fibration formula in transversal case}, we have a diagram
	\begin{equation}
		\begin{tikzcd}[column sep=-20,row sep=5]
			&&X\arrow[rd,equal]\arrow[dd,"\mathrm{coev}_{\mathcal{F},X/Y}"]\\
			&&&X\arrow[dd,"\mathrm{coev}_{\mathcal{F},X/S}"]\\
			\mathcal{F}\otimes_YD(\mathcal{F}/Y)\arrow[rr]\arrow[dd,"\mathrm{ev}_{\mathcal{F},X/Y}"swap]\arrow[dr]&&\mathcal{H}om_Y(\mathcal{F},\mathcal{F})\arrow[dr,equal]\\
			&\mathcal{H}om_{Y\times_SY}(Y,\mathcal{F}\otimes_SD(\mathcal{F}/S))\arrow[rr,equal]\arrow[dd,"\mathrm{ev}_{\mathcal{F},X/S}"swap]\arrow[dr]&&\mathcal{H}om_{Y\times_SY}(Y,\mathcal{H}om_S(\mathcal{F},\mathcal{F}))\\
			D(X/Y)\arrow[dr]&&\Delta_{Y/S}(\mathcal{F}\otimes_SD(\mathcal{F}/S))\arrow[dd]\\
			&D(X/S)\arrow[dr]\\
			&&D(X/Y/S).
		\end{tikzcd}
	\end{equation}
	
	This diagram defines a class $C_{X/Y/S}(\mathcal{F})\in H^0(X,\mathcal{K}_{X/Y/S})$, which is called the non-localized non-acyclicity class. Unwinding the definitions, this construction recovers \cite[(4.12)]{YZ25}.
\end{definition}

\begin{lemma}
	For $j:U\hookrightarrow X$ in $\mathrm{Sch}_S$ open immersion, we have natural diagram
	\begin{equation}
		\begin{tikzcd}
			X\arrow[r,two heads,"{[j^*]}"]\arrow[d,"C_{X/Y/S}(\mathcal{F})"swap]&U\arrow[d,"C_{U/Y/S}(\mathcal{F}|_U)"]\\
			D(X/Y/S)\arrow[r,two heads,"{[(j\times j)^*]}"]&D(U/Y/S).
		\end{tikzcd}
	\end{equation}
\end{lemma}
\begin{proof}
	We have the following diagram
	\begin{equation}
		\begin{tikzcd}
			X\arrow[d,"\mathrm{coev}_{\mathcal{F},X/Y}"]\arrow[r,two heads,"{[j^*]}"]&U\arrow[d,"\mathrm{coev}_{\mathcal{F}|_U,U/Y}"]\\
			\mathcal{H}om_Y(\mathcal{F},\mathcal{F})\arrow[r,two heads,"{[(j\times j)^*]}"]\arrow[d,equal]&\mathcal{H}om_Y(\mathcal{F}|_U,\mathcal{F}|_U)\arrow[d,equal]\\
			\mathcal{H}om_{Y\times_SY}(Y,\mathcal{F}\otimes_SD(\mathcal{F}/S))\arrow[r,two heads,"{[(j\times j)^*]}"]\arrow[d,"\ref{def: triangle functor}"]&\mathcal{H}om_{Y\times_SY}(Y,\mathcal{F}|_U\otimes_SD(\mathcal{F}|_U/S))\arrow[d,"\ref{def: triangle functor}"]\\
			\Delta_{Y/S}(\mathcal{F}\otimes_SD(\mathcal{F}/S))\arrow[r,two heads,"{[(j\times j)^*]}"]\arrow[d,"\mathrm{ev}_{\mathcal{F},X/S}"]&\Delta_{Y/S}(\mathcal{F}|_U\otimes_SD(\mathcal{F}|_U,S))\arrow[d,"\mathrm{ev}_{\mathcal{F}|_U,U/S}"]\\
			D(X/Y/S)\arrow[r,two heads,"{[j^*]}"]&D(U/Y/S).
		\end{tikzcd}
	\end{equation}
\end{proof}

\begin{definition}\label{def: NA class}
	Let $j:U\hookrightarrow X$ be an open immersion with closed complement $i:Z\hookrightarrow X$. Assume that $(X;\mathcal{F})\in\mathrm{CohCorr}_Y$ is dualizable in $\mathrm{CohCorr}_S$, and $\mathcal{F}|_U$ is dualizable in $\mathrm{CohCorr}_Y$.
	
	Let $D_Z(X/Y/S)=i_*i^!D(X/Y/S)$.
	By \eqref{eq: compatible dualizability}, $\Delta_{Y/S}(\mathcal{F}|_U\otimes_S\mathcal{H}om_S(\mathcal{F}|_U,S))=0$.
	
	By excision, we have a diagram
	\begin{equation}
		\begin{small}
			\begin{tikzcd}
				&X\arrow[d]\arrow[r]&j_*U\arrow[d]\\
				i_*i^!\Delta_{Y/S}(\mathcal{F}\otimes_SD(\mathcal{F}/S))\arrow[r,equal]\arrow[d]&\Delta_{Y/S}(\mathcal{F}\otimes_SD(\mathcal{F}/S))\arrow[r]\arrow[d]&j_*\Delta_{Y/S}(\mathcal{F}|_U\otimes_SD(\mathcal{F}|_U,S))=0\arrow[d]\\
				D_Z(X/Y/S)\arrow[r]&D(X/Y/S)\arrow[r]&j_*D(U/Y/S).
			\end{tikzcd}
		\end{small}
	\end{equation}
	
	This diagram defines a class $C_{X/Y/S}^Z(\mathcal{F})\in H^0_Z(X,\mathcal{K}_{X/Y/S})$, which recovers \cite[(4.14)]{YZ25}.
\end{definition}

\section{Exact nine-diagrams}\label{sec: Nine-diagram}
In this section, we study a special kind of diagrams called nine-diagrams. This will play an important role in our proof of the additivity theorem.

\subsection{}\label{subsec: exact triangle}
Let $\mathcal{C}$ be a stable category with zero object $0$.
\begin{enumerate}[label=(\alph*)]
	\item A square in $\mathcal{C}$ is a functor $M:\square\to\mathcal{C}$.
	\begin{equation}\label{eq: square}
		M:\begin{tikzcd}
			M(0,0)\arrow[r]\arrow[d]&M(0,1)\arrow[d]\\
			M(1,0)\arrow[r]&M(1,1).
		\end{tikzcd}
	\end{equation}
	\item $M$ is an exact square if it is both a pullback and a pushout square.
	\item $M$ is a triangle if $M(1,0)=0$. For simplicity, we denote the triangle
	\begin{equation}\label{eq: triangle}
		M:\begin{tikzcd}
			X\arrow[r]\arrow[d]&Y\arrow[d]\\
			0\arrow[r]&Z
		\end{tikzcd}
	\end{equation}
	by $(X\to Y\to Z)$ if no confusion will be caused.
	\item An exact triangle is a triangle which is exact as a square. In the derived category $\mathcal{D}(\mathcal{A})$ of an abelian category $\mathcal{A}$, exact triangles are given by distinguished triangles.
\end{enumerate} 

Let $\mathrm{ExTri}(\mathcal{C}),\mathrm{ExSq}(\mathcal{C}),\mathrm{Tri}(\mathcal{C})\subset\mathrm{Sq}(\mathcal{C})=\mathrm{Fun}(\square,\mathcal{C})$ be full subcategories spanned by exact triangles, exact squares and triangles respectively. Note they are all stable by \cite[1.1.3.2]{HA}.

The following lemma should be well-known to experts. We include a proof for the sake of completeness.

\begin{lemma}\label{lem: FD functor}
	There is a natural functor
	\begin{equation}
		\mathrm{FD}:\mathrm{ExSq}(\mathcal{C})\to\mathrm{ExTri}(\mathcal{C}),
	\end{equation}
	
	sending an exact square \eqref{eq: square} to
	\begin{equation}\label{eq: folded triangle}
		\begin{tikzcd}
			X\arrow[dr,phantom,"\square"]\arrow[r,"{(f,g)}"]\arrow[d]&Y\oplus Z\arrow[d,"{(p,-q)}"]\\
			0\arrow[r]&W.
		\end{tikzcd}
	\end{equation}
\end{lemma}
\begin{proof}
	The functor is given by the following composition
	\begin{equation}
		\mathrm{FD}:\mathrm{ExSq}(\mathcal{C})=\mathrm{Fun}([1]\coprod_{\{0\}}[1],\mathcal{C})\to\mathrm{Fun}([1],\mathcal{C})=\mathrm{ExTri}(\mathcal{C}),
	\end{equation}
	where the equivalences follow from \cite[4.3.2.15]{HTT}, and the middle functor sends $(Z\xleftarrow{g}X\xrightarrow{f}Y)$ to $(X\xrightarrow{(f,g)}Y\oplus Z)$.
	
	To see that the exact triangle obtained is of the form \eqref{eq: folded triangle}, we use the following calculation:
	\begin{align}
		W&=Y\coprod_ZX\\
		&=\mathrm{cofib}(0\to Y)\coprod_{\mathrm{cofib}(X[-1]\to0)}\mathrm{cofib}(0\to Z)\\
		&=\mathrm{cofib}\left(0\coprod_{X[-1]}0\to Y\coprod_0Z\right)\\
		&=\mathrm{cofib}(X\to Y\oplus Z).
	\end{align}
\end{proof}

\begin{remark}\label{rmk: sign convention}
	The minus sign in \eqref{eq: folded triangle} comes from transposition of the square (c.f. \cite[1.1.2.10]{HA})
	\begin{equation}
		\begin{tikzcd}
			X[-1]\arrow[r]\arrow[d]&0\arrow[d]\\
			0\arrow[r]&X.
		\end{tikzcd}
	\end{equation}
	
	There are several variants of \eqref{eq: folded triangle} with different sign conventions.
\end{remark}

\begin{definition}\label{def: nine-diagram}
	We call $\mathrm{Nine}(\mathcal{C})=\mathrm{Tri}(\mathrm{Tri}(\mathcal{C}))$ the category of nine-diagrams, whose objects are depicted as follows
	\begin{equation}\label{eq: nine-diagram}
		\begin{tikzcd}
			X_{0,0}\arrow[d,"d^v_{0,0}"]\arrow[r,"d^h_{0,0}"]&X_{0,1}\arrow[d,"d^v_{0,1}"]\arrow[r,"d^v_{0,1}"]&X_{0,2}\arrow[d,"d^v_{0,2}"]\\
			X_{1,0}\arrow[d,"d^v_{1,0}"]\arrow[r,"d^h_{1,0}"]&X_{1,1}\arrow[d,"d^v_{1,1}"]\arrow[r,"d^h_{1,1}"]&X_{1,2}\arrow[d,"d^v_{1,2}"]\\
			X_{2,0}\arrow[r,"d^h_{2,0}"]&X_{2,1}\arrow[r,"d^h_{2,1}"]&X_{2,2}.
		\end{tikzcd}
	\end{equation}
	
	Let $\mathrm{ExNine}(\mathcal{C})=\mathrm{ExTri}(\mathrm{ExTri}(\mathcal{C}))\subset\mathrm{Nine}(\mathcal{C})$ the full subcategory of exact nine-diagrams.
\end{definition}

\subsection{}\label{subsec: half nine-diagram}
Let $\mathrm{LowNine}(\mathcal{C})=\mathrm{Tri}(\mathcal{C})\times_{\mathrm{Fun}([1],\mathcal{C})}\mathrm{Sq}(\mathcal{C})\times_{\mathrm{Fun}([1],\mathcal{C})}\mathrm{Tri}(\mathcal{C})$ be the category of lower nine-diagrams, whose objects are depicted as follows
\begin{equation}
	\begin{tikzcd}
		&&X_{0,2}\arrow[d]\\
		&X_{1,1}\arrow[r]\arrow[d]&X_{1,2}\arrow[d]\\
		X_{2,0}\arrow[r]&X_{2,1}\arrow[r]&X_{2,2}.
	\end{tikzcd}
\end{equation}

Let $\mathrm{ExLowNine}(\mathcal{C})=\mathrm{ExTri}(\mathcal{C})\times_{\mathrm{Fun}([1],\mathcal{C})}\mathrm{Sq}(\mathcal{C})\times_{\mathrm{Fun}([1],\mathcal{C})}\mathrm{ExTri}(\mathcal{C})\subset\mathrm{LowNine}(\mathcal{C})$ be the full subcategory of exact lower nine-diagrams.

By \cite[4.3.2.15]{HTT}, the natural restriction functor
\begin{equation}
	\mathrm{Res}_{\mathrm{Low}}:\mathrm{Nine}(\mathcal{C})\to\mathrm{LowNine}(\mathcal{C})
\end{equation}
restricts to an equivalence
\begin{equation}\label{eq: extend half nine-diagram}
	\mathrm{ExNine}(\mathcal{C})=\mathrm{ExLowNine}(\mathcal{C}).
\end{equation}
	
Similar result is true for the category of (exact) upper nine-diagrams $\mathrm{(Ex)UpNine}(\mathcal{C})$.

\begin{proposition}\label{prop: nine-diagram to sequence}
	There is a natural functor
	\begin{equation}
		\mathrm{ExNine}(\mathcal{C})\to\mathrm{ExTri}(\mathcal{C}),
	\end{equation}
	sending an exact nine-diagram \eqref{eq: nine-diagram} to
	\begin{equation}\label{eq: key triangle}
		\mathrm{cofib}(d_0)\to C_2\to\mathrm{fib}(d_3),
	\end{equation}
	where the chain $C_0\xrightarrow{d_0}C_1\xrightarrow{d_1}C_2\xrightarrow{d_2}C_3\xrightarrow{d_3}C_4$ is given by
	\begin{equation}\label{eq: sequence def}
		\begin{aligned}
			&C_i=\oplus_{j+k=i}X_{j,k},\\
			&d_i|_{X_{j,k}}=d^h_{j,k}+(-1)^{j+k}d^v_{j,k}.
		\end{aligned}
	\end{equation}
\end{proposition}
\begin{remark}
	See \ref{rmk: sign convention} for remarks on sign conventions. The signs of $d_i$ we chosen here can be depicted as follows
	\begin{equation}
		\begin{tikzcd}
			X_{0,0}\arrow[r,"+"]\arrow[d,"+"]&X_{0,1}\arrow[r,"+"]\arrow[d,"-"]&X_{0,2}\arrow[d,"+"]\\
			X_{1,0}\arrow[r,"+"]\arrow[d,"-"]&X_{1,1}\arrow[r,"+"]\arrow[d,"+"]&X_{1,2}\arrow[d,"-"]\\
			X_{2,0}\arrow[r,"+"]&X_{2,1}\arrow[r,"+"]&X_{2,2}.
		\end{tikzcd}
	\end{equation}
\end{remark}
\begin{proof}
	By \ref{lem: FD functor} we have an exact square
	\begin{equation}\label{eq: cofib d_0}
		\begin{tikzcd}
			X_{0,0}\arrow[rd,phantom,"\square"]\arrow[r]\arrow[d]&X_{0,1}\arrow[d]\\
			X_{1,0}\arrow[r]&\mathrm{cofib}(d_0),
		\end{tikzcd}
	\end{equation}
	which induces the following exact squares
	\begin{equation}\label{eq: cofib d_0 induces}
		\begin{tikzcd}
			X_{0,0}\arrow[dr,phantom,"\square"]\arrow[d]\arrow[r]&X_{0,1}\arrow[d]\\
			X_{1,0}\arrow[r]\arrow[d]\arrow[dr,phantom,"\square"]&\mathrm{cofib}(d_0)\arrow[dr,phantom,"\square"]\arrow[r]\arrow[d]&X_{1,1}\arrow[d]\\
			0\arrow[r]&X_{0,2}\arrow[r]&X_{1,2}.
		\end{tikzcd}
	\end{equation}
	
	By \ref{lem: FD functor} (in fact a variant with different sign convention) we have an exact square
	\begin{equation}\label{eq: fib d_3}
		\begin{tikzcd}
			\mathrm{fib}(d_3)\arrow[r]\arrow[d]\arrow[rd,phantom,"\square"]&X_{1,2}\arrow[d]\\
			X_{2,1}\arrow[r]&X_{2,2},
		\end{tikzcd}
	\end{equation}
	which induces the following exact cube (i.e. every face is exact)
	\begin{equation}\label{eq: fib d_3 induces}
		\begin{tikzcd}[row sep=10,column sep=10]
			X_{0,2}\oplus X_{2,0}\arrow[rr]\arrow[dd]\arrow[dr]&&X_{0,2}\arrow[dd]\arrow[dr]\\
			&X_{2,0}\arrow[rr,crossing over]&&0\arrow[dd]\\
			\mathrm{fib}(d_3)\arrow[dr]\arrow[rr]&&X_{1,2}\arrow[dr]\\
			&X_{2,1}\arrow[uu,leftarrow,crossing over]\arrow[rr]&&X_{2,2}.
		\end{tikzcd}
	\end{equation}
	
	Combine the lower right square of \eqref{eq: cofib d_0 induces} and the back face of \eqref{eq: fib d_3 induces}, we obtain
	\begin{equation}
		\begin{tikzcd}
			\mathrm{cofib}(d_0)\arrow[r]\arrow[d]\arrow[dr,phantom,"\square"]&X_{0,2}\oplus X_{2,0}\arrow[r]\arrow[d]\arrow[rd,phantom,"\square"]&X_{0,2}\arrow[d]\\
			X_{1,1}\arrow[r]&\mathrm{fib}(d_3)\arrow[r]&X_{1,2}.
		\end{tikzcd}
	\end{equation}
	
	We obtain \eqref{eq: key triangle} by applying \ref{lem: FD functor} again. The naturality is obvious (c.f. \cite[4.3.2.15]{HTT}).
\end{proof}

\begin{example}\label{ex: source nine-diagram}
	For $X\in\mathcal{C}$, let $\mathrm{S}(X)$ be the exact nine-diagram
	\begin{equation}
		\begin{tikzcd}
			X[-1]\arrow[r]\arrow[d]&0\arrow[r]\arrow[d]&X\arrow[d,"\mathrm{id}"]\\
			0\arrow[r]\arrow[d]&X\arrow[r,"\mathrm{id}"]\arrow[d,"\mathrm{id}"]&X\arrow[d]\\
			X\arrow[r,"\mathrm{id}"]&X\arrow[r]&0.
		\end{tikzcd}
	\end{equation}
	Its associated exact triangle \eqref{eq: key triangle} is
	\begin{equation}
		X\xrightarrow{(\mathrm{id},-\mathrm{id},\mathrm{id})}X\oplus X\oplus X\xrightarrow{\binom{\mathrm{id},\mathrm{id},0}{0,\mathrm{id},\mathrm{id}}}X\oplus X.
	\end{equation}
\end{example}

\begin{example}\label{ex: target nine-diagram}
	For $X\in\mathcal{C}$, let $\mathrm{T}(X)$ be the exact nine-diagram
	\begin{equation}
		\begin{tikzcd}
			0\arrow[r]\arrow[d]&X\arrow[r,"\mathrm{id}"]\arrow[d,"\mathrm{id}"]&X\arrow[d]\\
			X\arrow[r,"\mathrm{id}"]\arrow[d,"\mathrm{id}"]&X\arrow[r]\arrow[d]&0\arrow[d]\\
			X\arrow[r]&0\arrow[r]&X[1].
		\end{tikzcd}
	\end{equation}
	Its associated exact triangle \eqref{eq: key triangle} is
	\begin{equation}
		X\oplus X\xrightarrow{\binom{\mathrm{id},-\mathrm{id},0}{0,\mathrm{id},-\mathrm{id}}}X\oplus X\oplus X\xrightarrow{(\mathrm{id},\mathrm{id},\mathrm{id})}X.
	\end{equation}
\end{example}

\section{Proof of the additivity theorem}\label{sec: Additivity}
In this section, we use the theory of exact nine-diagrams (especially proposition \ref{prop: nine-diagram to sequence}) to prove the additivity of non-acyclicity classes.

\subsection{}
Let $h:X\xrightarrow{f}Y\xrightarrow{g}S$ in $\mathrm{Sch}_S$, $\mathcal{F},\mathcal{G}\in\mathcal{D}(X)$. Assume $g$ satisfies assumption \ref{assum: smooth assumption}.

For $\alpha:\mathcal{F}\to\mathcal{G}$, we have a coevaluation map $\mathrm{coev}_{\alpha,X/Y}:X\to\mathcal{H}om_Y(\mathcal{F},\mathcal{G})$ as follows
\begin{equation}
	\begin{tikzcd}
		X\arrow[r,"\mathrm{coev}_{\mathcal{F},X/Y}"]\arrow[d,"\mathrm{coev}_{\mathcal{G},X/Y}"]&\mathcal{H}om_Y(\mathcal{F},\mathcal{F})\arrow[d,"{\alpha}"]\\
		\mathcal{H}om_Y(\mathcal{G},\mathcal{G})\arrow[r,"{\alpha}"]&\mathcal{H}om_Y(\mathcal{F},\mathcal{G}).
	\end{tikzcd}
\end{equation}

For $\beta:\mathcal{G}\to\mathcal{F}$, we have an evaluation map $\mathrm{ev}_{\beta,X/Y}:\mathcal{G}\otimes_Y\mathcal{H}om_Y(\mathcal{F},Y)\to D(X/Y)$ as follows
\begin{equation}
	\begin{tikzcd}
		\mathcal{G}\otimes_Y\mathcal{H}om_Y(\mathcal{F},Y)\arrow[r,"\beta"]\arrow[d,"\beta"]&\mathcal{F}\otimes_Y\mathcal{H}om_Y(\mathcal{F},Y)\arrow[d,"\mathrm{ev}_{\mathcal{F},X/Y}"]\\
		\mathcal{G}\otimes_Y\mathcal{H}om_Y(\mathcal{G},Y)\arrow[r,"\mathrm{ev}_{\mathcal{G},X/Y}"]&D(X/Y).
	\end{tikzcd}
\end{equation}

Recall that if $\mathcal{F}$ is dualizable in $\mathrm{CohCorr}_Y$, the Verdier pairing $\langle\alpha,\beta\rangle_{X/Y}\in H^0(X,\mathcal{K}_{X/S})$ is given by (c.f. \cite{LZ22})
\begin{equation}
	X\xrightarrow{\mathrm{coev}_{\alpha,X/Y}}\mathcal{H}om_Y(\mathcal{F},\mathcal{G})=\mathcal{G}\otimes_Y\mathcal{H}om_Y(\mathcal{F},Y)\xrightarrow{\mathrm{ev}_{\beta,X/Y}}X.
\end{equation}

\begin{definition}
	For $\mathcal{F}$ dualizable in $\mathrm{CohCorr}_S$, define a pairing
	\begin{equation}
		\langle\alpha,\beta\rangle_{X/Y/S}\in H^0(X,\mathcal{K}_{X/Y/S})
	\end{equation}
	as the following composition
	\begin{equation}
		\begin{aligned}
			&\quad X\xrightarrow{\mathrm{coev}_{\alpha,X/Y}}\mathcal{H}om_Y(\mathcal{F},\mathcal{G})=\mathcal{H}om_{Y\times_SY}(Y,\mathcal{H}om_S(\mathcal{F},\mathcal{G}))\\
			&=\mathcal{H}om_{Y\times_SY}(Y,\mathcal{G}\otimes_S\mathcal{H}om_S(\mathcal{F},S))\xrightarrow{\ref{def: triangle functor}}\Delta_{Y/S}(\mathcal{G}\otimes_SD(\mathcal{F}/S))\\
			&\xrightarrow{\mathrm{ev}_{\beta,X/S}}\Delta_{Y/S}D(X/S)=D(X/Y/S).
		\end{aligned}
	\end{equation}
	
	In particular $\langle\mathrm{id}_{\mathcal{F}},\mathrm{id}_{\mathcal{F}}\rangle_{X/Y/S}=C_{X/Y/S}(\mathcal{F})$ (c.f. \ref{def: non-localized NA class}).
\end{definition}

\begin{proposition}\label{prop: additivity of non-localized NA class}
	Let $E_{\mathcal{F}}=(\mathcal{F}'\to\mathcal{F}\to\mathcal{F}''),E_{\mathcal{G}}=(\mathcal{G}'\to\mathcal{G}\to\mathcal{G}'')$ be exact triangles in $\mathcal{D}(X)$. Assume that $\mathcal{F}',\mathcal{F},\mathcal{F}''$ are dualizable in $\mathrm{CohCorr}_S$.
	
	Let $(\alpha',\alpha,\alpha''):E_{\mathcal{F}}\to E_{\mathcal{G}}$ be a map of exact triangles depicted as follows
	\begin{equation}\label{eq: alpha F to G}
		\begin{tikzcd}
			\mathcal{F}'\arrow[r,"\alpha'"]\arrow[d]&\mathcal{G}'\arrow[d]\\
			\mathcal{F}\arrow[r,"\alpha"]\arrow[d]&\mathcal{G}\arrow[d]\\
			\mathcal{F}''\arrow[r,"\alpha''"]&\mathcal{G}'',
		\end{tikzcd}
	\end{equation}
	and $(\beta',\beta,\beta''):E_{\mathcal{G}}\to E_{\mathcal{F}}$ a map of exact triangles depicted as follows
	\begin{equation}\label{eq: beta G to F}
		\begin{tikzcd}
			\mathcal{G}'\arrow[r,"\beta'"]\arrow[d]&\mathcal{F}'\arrow[d]\\
			\mathcal{G}\arrow[r,"\beta"]\arrow[d]&\mathcal{F}\arrow[d]\\
			\mathcal{G}''\arrow[r,"\beta''"]&\mathcal{F}''.
		\end{tikzcd}
	\end{equation}
	
	Then we have an equality in $H^0(X,\mathcal{K}_{X/Y/S})$
	\begin{equation}
		\langle\alpha,\beta\rangle=\langle\alpha',\beta'\rangle+\langle\alpha'',\beta''\rangle.
	\end{equation}
	
	In particular, for $\mathcal{F}\in\mathrm{CohCorr}_S$ dualizable, $C_{X/Y/S}(\mathcal{F})$ is additive in $\mathcal{F}$.
\end{proposition}
\begin{proof}
	By adjunction, \eqref{eq: alpha F to G} induces a map of exact lower nine-diagrams
	\begin{equation}
		\begin{tikzcd}[row sep=5,column sep=5]
			&&&&Y\arrow[dr]\arrow[dd]\\
			&&&&&\mathcal{H}om_Y(\mathcal{F}'',\mathcal{G}'')\arrow[dd]\\
			&&Y\arrow[rr]\arrow[dd]\arrow[dr]&&Y\arrow[dr]\arrow[dd]\\
			&&&\mathcal{H}om_Y(\mathcal{F},\mathcal{G})\arrow[rr]\arrow[dd]&&\mathcal{H}om_Y(\mathcal{F},\mathcal{G}'')\arrow[dd]\\
			Y\arrow[rr]\arrow[dr]&&Y\arrow[rr]\arrow[dr]&&0\arrow[dr]\\
			&\mathcal{H}om_Y(\mathcal{F}',\mathcal{G}')\arrow[rr]&&\mathcal{H}om_Y(\mathcal{F}',\mathcal{G})\arrow[rr]&&\mathcal{H}om_Y(\mathcal{F}',\mathcal{G}''),
		\end{tikzcd}
	\end{equation}
	which extends (uniquely) to a map of exact nine-diagrams by \ref{eq: extend half nine-diagram}
	\begin{equation}
		\mathrm{coev}_{E_\alpha}:\mathrm{S}(Y)\to\mathcal{H}om_Y(E_{\mathcal{F}},E_{\mathcal{G}}).
	\end{equation}
	
	The map $\mathrm{coev}_{E_{\mathcal{F}}}$ lies over (the constant functor of value) the correspondence $[(\delta_f)_!f^*]$, and hence factors through the cocartesian edge over $[f^*]$:
	\begin{equation}
		\begin{tikzcd}
			\mathrm{S}(Y)\arrow[dr,"\mathrm{coev}_{E_\alpha}"swap]\arrow[r,two heads,"{[f^*]}"]&\mathrm{S}(X)\arrow[d,dashed,"\mathrm{coev}_{E_\alpha,X/Y}"]\\
			&\mathcal{H}om_Y(E_{\mathcal{F}},E_{\mathcal{F}}).
		\end{tikzcd}
	\end{equation}
	
	Similarly, \eqref{eq: beta G to F} induces a diagram of exact nine-diagrams
	\begin{equation}
		\begin{tikzcd}
			E_{\mathcal{G}}\otimes_S\mathcal{H}om_S(E_{\mathcal{F}},S)\arrow[r,dashed,"\mathrm{ev}_{E_\beta,X/S}"]\arrow[dr,"\mathrm{ev}_{E_\beta}"swap]&\mathrm{T}(D(X/S))\arrow[d,hook,"{[h_!]}"]\\
			&\mathrm{T}(S).
		\end{tikzcd}
	\end{equation}
	
	We now have the following maps of exact nine-diagrams
	\begin{equation}
		\begin{aligned}
			&\mathrm{S}(X)\xrightarrow{\mathrm{coev}_{E_\alpha,X/Y}}\mathcal{H}om_Y(E_{\mathcal{F}},E_{\mathcal{G}})=\mathcal{H}om_{Y\times_SY}(Y,E_{\mathcal{G}}\otimes_S\mathcal{H}om_S(E_{\mathcal{F}},S))\\
			\xrightarrow{\ref{def: triangle functor}}&\Delta_{Y/S}(E_{\mathcal{G}}\otimes_S\mathcal{H}om_S(E_{\mathcal{F}},S))\xrightarrow{\mathrm{ev}_{E_\beta,X/S}}\mathrm{T}(\Delta_{Y/S}D(X/S))=\mathrm{T}(D(X/Y/S)).
		\end{aligned}
	\end{equation}
	
	Passing to the associated exact triangles \eqref{eq: key triangle}, we conclude that
	\begin{equation}
		\langle\alpha',\beta'\rangle-\langle\alpha,\beta\rangle+\langle\alpha'',\beta''\rangle=0.
	\end{equation}
\end{proof}

\begin{corollary}\label{cor: additivity of NA class}
	Let $E_{\mathcal{F}}=(\mathcal{F}'\to\mathcal{F}\to\mathcal{F}'')$ be an exact triangle in $\mathcal{D}(X)$. Let $U\hookrightarrow X$ open immersion with closed complement $Z\hookrightarrow X$. Assume that $\mathcal{F}',\mathcal{F},\mathcal{F}''$ are dualizable in $\mathrm{CohCorr}_S$, and $\mathcal{F}'|_U,\mathcal{F}|_U,\mathcal{F}''|_U$ are dualizable in $\mathrm{CohCorr}_Y$.
	
	Then we have an equality in $H^0_Z(X,\mathcal{K}_{X/Y/S})$
	\begin{equation}
		C_{X/Y/S}^Z(\mathcal{F})=C_{X/Y/S}^Z(\mathcal{F}')+C_{X/Y/S}^Z(\mathcal{F}'').
	\end{equation}
\end{corollary}
\begin{proof}
	Proceed as in the proof of \ref{prop: additivity of non-localized NA class}, we have the following maps of exact nine-diagrams
	\begin{equation}
		\begin{aligned}
			\mathrm{S}(X)&\xrightarrow{\mathrm{coev}_{E_{\mathrm{id}},X/Y}}\mathcal{H}om_Y(E_{\mathcal{F}},E_{\mathcal{F}})\to\Delta_{Y/S}(E_{\mathcal{F}}\otimes_S\mathcal{H}om_S(E_{\mathcal{F}},S))\\
			&=i_*i^!\Delta_{Y/S}(E_{\mathcal{F}}\otimes_S\mathcal{H}om_S(E_{\mathcal{F}},S))\xrightarrow{\mathrm{ev}_{E_{\mathrm{id}},X/S}}\mathrm{T}(D_Z(X/Y/S)).
		\end{aligned}
	\end{equation}
	
	We conclude by passing to the associated exact triangles \eqref{eq: key triangle}.
\end{proof}

\end{document}